\documentclass[letter, 10pt, conference]{ieeeconf}
\IEEEoverridecommandlockouts        

\usepackage{etex}
\usepackage[vlined,ruled]{algorithm2e}
\usepackage[dvipsnames]{xcolor}
\usepackage{pgfpages}
\usepackage[cmex10]{amsmath}		
\usepackage{amssymb, mathrsfs}
\usepackage{cite}
\usepackage{url}

\usepackage{dsfont}
\usepackage{siunitx}
\usepackage{verbatim}									
\usepackage{arydshln}
\usepackage{mathtools}									
\usepackage{siunitx}									
\usepackage{mdframed}

\usepackage{graphicx}
\usepackage{epstopdf}

\usepackage[font=small]{caption}
\usepackage{subcaption}
\usepackage[activate={true,nocompatibility},final,tracking=true,kerning=true,spacing=true,factor=1100,stretch=10,shrink=10]{microtype}

\usepackage{array}
\usepackage{multirow}
\usepackage{enumerate}
\usepackage{booktabs}

\graphicspath{{./fig/}}

\usepackage{pgfplots}
\pgfplotsset{width=8cm,compat=newest}
\newcommand*{\damping}{0.006}%
\newcommand*{\freq}{25}%
\pgfmathsetmacro{\freqd}{sqrt(1-(\damping)^2)*\freq}%

\pgfplotsset{
    standard/.style={
    axis x line=middle,
    axis y line=middle,
    enlarge x limits=0.15,
	enlarge y limits=0.15,
	every axis plot post/.style={mark options={fill=black}},
	}
}

\pgfplotsset{%
    ,compat=1.12
    ,every axis x label/.style={at={(current axis.right of origin)},anchor=north west}
    ,every axis y label/.style={at={(current axis.above origin)},anchor=north east}
    }

\usepackage{tikz}
\usetikzlibrary{arrows}
\usetikzlibrary{decorations.pathmorphing}
\usetikzlibrary{decorations.markings}
\usetikzlibrary{snakes}
\usetikzlibrary{matrix}
\usetikzlibrary{calc}
\usetikzlibrary{patterns}
\usetikzlibrary{positioning}
\usetikzlibrary{shapes}
\usetikzlibrary{shapes.symbols}
\usetikzlibrary{shapes.geometric}
\usetikzlibrary{shadows.blur}
\usetikzlibrary{shapes.arrows}
\usetikzlibrary{shapes.multipart}
\usetikzlibrary{shapes.callouts}
\usetikzlibrary{shapes.misc}

\tikzstyle{every node}=[font=\small]
\tikzstyle{every path}=[line width=0.8pt,line cap=round,line join=round]

\usepackage[american,smartlabels]{circuitikz}
\ctikzset{bipoles/thickness=1}
\ctikzset{bipoles/length=0.8cm}
\ctikzset{bipoles/diode/height=.375}
\ctikzset{bipoles/diode/width=.3}
\ctikzset{tripoles/thyristor/height=.8}
\ctikzset{tripoles/thyristor/width=1}
\ctikzset{bipoles/vsourceam/height/.initial=.7}
\ctikzset{bipoles/vsourceam/width/.initial=.7}

\newcommand{\real}{\mathbb{R}}

\newcommand{\setdef}[2]{\{#1 \;|\; #2\}}

\newcommand{\inner}[3]{\ensuremath{\langle #1,#2\rangle}_{#3}}
\DeclareMathOperator*{\minimize}{minimize} 									
\DeclareMathOperator{\Proj}{Proj}
\DeclareMathOperator*{\argmin}{argmin}

\DeclareSymbolFont{bbold}{U}{bbold}{m}{n}
\DeclareSymbolFontAlphabet{\mathbbold}{bbold}

\newcommand{\map}[3]{#1: #2 \rightarrow #3}

\renewcommand{\theenumi}{(\roman{enumi})}

\newcommand{\tb}{\color{blue}}

\usepackage{soul}

\newcommand{\define}{\coloneqq}

\newcommand\oprocendsymbol{\hbox{$\square$}}
\newcommand\oprocend{\relax\ifmmode\else\unskip\hfill\fi\oprocendsymbol}

\newtheorem{theorem}{Theorem}[section]

\newtheorem{proposition}[theorem]{Proposition}

\newtheorem{remark}{Remark}[section]

\newenvironment{pfof}[1]{\vspace{1ex}\noindent{\itshape Proof of
    #1:}\hspace{0.5em}} {\hfill\oprocend\vspace{1ex}}

\newif\ifforstudents

\let\labelindent\relax
\usepackage{enumitem}

\usepackage{cuted}

\newcommand{\T}{\mathsf{T}}

\renewcommand{\tb}{\color{black}}

\binoppenalty=\maxdimen
\relpenalty=\maxdimen

\title{\bf Low-Gain Stability of Projected Integral Control for Input-Constrained Discrete-Time Nonlinear Systems
\thanks{
}
}

\author{John W. Simpson-Porco%
    \thanks{
      J. W. Simpson-Porco is with the Department of Electrical and Computer Engineering, University of Toronto, 10 King's College Road,
Toronto, ON, M5S 3G4, Canada. Email: {\tt jwsimpson@ece.utoronto.ca}. Work supported by NSERC Discovery Grant RGPIN-2017-04008.
}
}

\begin{document}
\maketitle
\thispagestyle{empty}
\pagestyle{empty}


\begin{abstract}
We consider the problem of zeroing an error output of a nonlinear discrete-time system in the presence of constant exogenous disturbances, subject to hard convex constraints on the input signal. The design specification is formulated as a variational inequality, and we adapt a forward-backward splitting algorithm to act as an integral controller which ensures that the input constraints are met at each time step. We establish a low-gain stability result for the closed-loop system when the plant is exponentially stable, generalizing previously known results for integral control of discrete-time systems. Specifically, it is shown that if the composition of the plant equilibrium input-output map and the integral feedback gain is strongly monotone, then the closed-loop system is exponentially stable for all sufficiently small integral gains. {\tb The method is illustrated via application to a four-tank process.}
\end{abstract}


\section{Introduction}\label{Sec:Introduction}

It is a well-known control principle that regulation of an error signal to zero can be achieved robustly in the presence of model uncertainty and constant references/disturbances only through integral feedback control \cite{BAF-WMW:76}. The presence of control input constraints however presents challenges to traditional integral controller designs; sufficient actuator authority may not be available to achieve exact regulation for all references/disturbances, and dynamic performance is sometimes degraded through the so-called wind-up phenomenon \cite{KJA-LR:89}.


There are two broad approaches for accomodating limited actuator authority. The explicit approach is to directly include input constraints into the design, as done in receding-horizon control \cite{JBR-DQM:09}, bounded integral control \cite{GCK-QCZ-BR-MK:16}, and in other nonlinear/adaptive approaches \cite{DEM-EJD:93}. {\tb A more traditional approach is to first design ignoring the actuator limits, and subsequently augment or retro-fit the design in order to improve performance in the presence of saturation; this category would include both classic and modern anti-windup design \cite{ST-MT:09,RSC-JVF-ATS-JMGDSJ:21}, and reference modification \cite{IK-EG-SDC:14,PS-SAB-CD-SDC:19}.}




Returning now to the fundamentals of integral control, a commonly encountered case in practice is that the system one wishes to control is complex, and limited dynamic model information is available, but it is however known that the system is stable (possibly achieved via a stabilizing controller design). A general and well-established design philosophy is that asymptotic tracking and disturbance rejection can be guaranteed by adding a supplementary integral control loop, and that the closed-loop stability will be guaranteed if the integral gain is sufficiently low; a famous and widely-deployed example of this design philosophy is the tuning of automatic generation control in power systems \cite{JWSP-NM:20m}. 

For finite-dimensional multi-input multi-output (MIMO) linear time-invariant (LTI) systems, the fundamental stability result for this approach is due to Davison \cite[Lemma 3]{EJD:76}; see also \cite[Theorem 3]{MM:85}. While \cite{EJD:76} is in continuous-time, the key result is identical in the discrete-time case \cite{AP-HNK:83}. Consider the plant model
\begin{equation}\label{Eq:LTI}
\begin{aligned}
x_{k+1} &= Ax_{k} + Bu_{k} + B_ww\\
e_{k} &= Cx_{k} + Du_{k} + D_ww
\end{aligned}
\end{equation}
with state $x \in \real^n$, control input $u \in \real^m$, constant disturbance/reference signal $w \in \real^{n_w}$, and error output $e \in \real^p$; we associate a sampling period $T_{\rm s} > 0$ with \eqref{Eq:LTI}. Assume that $A$ is Schur stable, and let $G(z) = C(zI_n-A)^{-1}B+D$ denote the transfer matrix of \eqref{Eq:LTI} from $u$ to $e$. One interconnects the system \eqref{Eq:LTI} with the integral controller
\begin{equation}\label{Eq:LTIIntegral}
\begin{aligned}
\eta_{k+1} &= \eta_{k} - \tfrac{T_{\rm s}}{T_{\rm i}} e_{k}, \qquad u_{k} = K\eta_{k},
\end{aligned}
\end{equation}
where $K \in \real^{m \times p}$ is a gain matrix and $T_{\rm i} > 0$ is the integral time constant. Davison's low-gain stability result states that if $-G(1)K$ is Hurwitz stable, then there exists $T_{\rm i}^{\star} > 0$ such that the closed-loop system is exponentially stable for all $T_{\rm i} \in (T_{\rm i}^{\star},\infty)$. {\tb A substantial literature exists on extensions of this core result to infinite-dimensional linear systems; see \cite{MEG-CG-HL:21} and the references therein.} Initial extensions to the continuous-time nonlinear case were given in \cite{CD-CL:85}. In \cite{JWSP:20a} the author further generalized these conditions via contraction theory, and provided a LMI-based procedure to design low-gain integral controllers for continuous-time nonlinear systems.



\smallskip

\emph{Contributions:} In this paper we further contribute to the study of constrained and low-gain integral control. {\tb We begin by proposing that the error-zeroing criteria in the presence of arbitrary convex input constraints be formulated as a variational inequality \cite{FF-JSP:03}}. This leads us to adopt a version of the \emph{projection} or \emph{forward-backward algorithm} \cite{FF-JSP:03b} as a {\tb novel} constrained integral controller. The design is in discrete-time, and is therefore immediately appropriate for digital control implementations. While this design explicitly enforces input constraints at each time instant, it has the following commonality with more traditional anti-windup approaches: if the input constraints are not encountered during operation, the scheme reduces to the classical integral controller \eqref{Eq:LTIIntegral}. Our main stability result (Theorem \ref{Thm:DPIC}) establishes that the ``low-gain integral control stability principle'' described above also holds for this projected integral controller, which extends the main result of \cite{JWSP:20a} to {\tb discrete-time nonlinear systems with input constraints and with our proposed constrained integral controller. To our knowledge, this constitutes the first comprehensive analysis of low-gain integral control for discrete-time nonlinear systems.  We illustrate the design by applying it to the nonlinear quadruple-tank model of \cite{KHJ:00}.}

\smallskip

\emph{Notation:} Given two vectors $x$ and $y$, $\mathrm{col}(x,y)$ denotes their vertical concatenation. If $P$ is a symmetric matrix $\lambda_{\rm min}(P)$ and $\lambda_{\rm max}(P)$ denote its minimum and maximum eigenvalues. A function $\map{f}{X}{\real^n}$ is Lipschitz continuous on $X \subseteq \real^n$ if there exists $L > 0$ such that $\|f(x)-f(y)\| \leq L \|x-y\|$ for all $x,y \in X$.


\section{Problem Formulation}
\label{Sec:Problem}

\subsection{Plant Model and Assumptions}

We consider a plant described by a finite-dimensional nonlinear time-invariant state-space model
\begin{equation}\label{Eq:Plant}
\begin{aligned}
x_{k+1} = f(x_k,u_k,w), \qquad e_k = h(x_k,u_k,w),
\end{aligned}
\end{equation}
where $x_k \in \real^n$ is the state, $u_k \in \real^m$ is the control input, and $w \in \real^{n_w}$ is a {\tb constant} exogenous signal (reference signals and/or disturbances). The signal $e_k \in \real^{p}$ with $p \leq m$ is an error output to be driven to zero. As the model \eqref{Eq:Plant} would most commonly arise via discretization of a continuous-time model, we associate a sampling period $T_{\rm s} > 0$ to \eqref{Eq:Plant}.

For any fixed $w$, the possible equilibrium state-input-error triplets $(\bar{x},\bar{u},\bar{e})$ are determined by the algebraic equations
\[
\begin{aligned}
\bar{x} &= f(\bar{x},\bar{u},w), \qquad \bar{e} = h(\bar{x},\bar{u},w).
\end{aligned}
\]
To capture the steady-state and dynamic behaviour of \eqref{Eq:Plant}, we assume that there exist convex sets $\mathcal{X},\mathcal{U},\mathcal{W}$ such that
\begin{enumerate}[label=(A\arabic*)]
\item \label{Ass:Plant-0} $f$ and $h$ are continuous in all arguments on $\mathcal{X} \times \mathcal{U} \times \mathcal{W}$, $f$ is continuously differentiable with respect to $x$ and $u$, and $f$, $h$, $\tfrac{\partial f}{\partial x}$, and $\tfrac{\partial f}{\partial u}$ are all Lipschitz continuous on $\mathcal{X} \times \mathcal{U}$ uniformly in $w \in \mathcal{W}$;
\item \label{Ass:Plant-1} 
 there is a class $C^1$ map $\map{\pi_{x}}{\mathcal{U}\times\mathcal{W}}{\mathcal{X}}$ which is Lipschitz continuous on $\mathcal{U} \times \mathcal{W}$ such that
\[
\pi_{x}(u,w) = f(\pi_{x}(u,w),u,w), \qquad (u,w) \in \mathcal{U} \times \mathcal{W}.
\]
\item \label{Ass:Plant-3} the equilibrium $\bar{x} = \pi_x(u,w)$ is exponentially stable, uniformly in $(u,w) \in \mathcal{U} \times \mathcal{W}$.
\end{enumerate}

\smallskip

Assumptions \ref{Ass:Plant-0}--\ref{Ass:Plant-3} capture the idea that the plant model is sufficiently smooth, and converges exponentially to a locally unique equilibrium when subject to reasonable constant inputs $u$ and $w$. {\tb The stability property may be inherent to the system, or may have been achieved through an initial stabilizing feedback design. While \ref{Ass:Plant-0}--\ref{Ass:Plant-3} cannot be expected to generically hold for nonlinear systems, many practical systems of interest \textemdash{} such as chemical process dynamics and electric power systems \textemdash{} are internally stable and/or are stabilized via inner-loop controller designs, and can hence be expected to satisfy \ref{Ass:Plant-0}--\ref{Ass:Plant-3}. We call}
\begin{equation}\label{Eq:DefofPi}
\map{\pi}{\mathcal{U}\times\mathcal{W}}{\real^m}, \quad \pi(\bar{u},w) \define h(\pi_{x}(\bar{u},w),\bar{u},w)
\end{equation}
the \emph{equilibrium input-to-error map}, which produces the equilibrium error $\bar{e} = \pi(\bar{u},w)$ associated with the constant control input $\bar{u}$ and disturbance $w$. When applied to the LTI system \eqref{Eq:LTI}, \ref{Ass:Plant-0}--\ref{Ass:Plant-3} simply reduce to $A$ being Schur stable, and the mapping $\pi$ becomes
\begin{equation}\label{Eq:LTIPi}
\pi(\bar{u},w) = G(1)\bar{u} + G_{w}(1)w,
\end{equation}
where $G_{w}(z) = C(zI_n-A)^{-1}B_w + D_w$.

\subsection{Constrained Error-Zeroing Specification}
\label{Sec:Error}


Let $\mathcal{C} \subseteq \mathcal{U}$ be a closed non-empty convex set which describes actuator limits. Following \eqref{Eq:LTIIntegral}, the control signal $u_k$ from our new integral controller will be generated as
\begin{equation}\label{Eq:LinearFeedback}
u_k = K\eta_k
\end{equation}
where $K \in \real^{m \times p}$ is a gain matrix to be designed and $\eta_k \in \real^p$ is the controller state.  It follows that the preimage
\[
\Gamma \define \setdef{\eta \in \real^p}{K\eta \in \mathcal{C}}
\]
is also a closed and non-empty convex set. For example, in applications $\mathcal{C}$ is often polyhedral, in which case so is $\Gamma$.

Our ideal design objective would be to ensure that for any disturbance $w \in \mathcal{W}$, the error signal $e_k$ is asymptotically driven to zero, and that the input constraint $u_k \in \mathcal{C}$ is satisfied at all times. As one might expect however, input constraints may prevent us from exactly zeroing the steady-state error $\bar{e} = \pi(\bar{u},w)$ for at least some disturbances $w \in \mathcal{W}$. We therefore relax the design objective, and instead seek an equilibrium value $\bar{\eta} \in \Gamma$ for the controller state such that
\begin{equation}\label{Eq:VI}
\inner{\bar{e}}{\eta-\bar{\eta}}{P} = \inner{\pi(K\bar{\eta},w)}{\eta-\bar{\eta}}{P} \geq 0, \qquad \forall \eta \in \Gamma,
\end{equation}
where $\inner{x}{y}{P} = x^{\T}Py$ is the inner product on $\real^p$ induced by some positive definite matrix $P \succ 0$. The inequality \eqref{Eq:VI} is called a \emph{variational inequality} \cite{FF-JSP:03}, and we notate a solution of the inequality as $\bar{\eta} \in \mathsf{VI}_{P}(\Gamma,\pi \circ K)$. Note that if $\bar{\eta}$ lies in the interior of the set $\Gamma$, then there exists $\tau > 0$ such that $\eta = \bar{\eta} - \tau \bar{e} \in \Gamma$. The inequality \eqref{Eq:VI} then implies that $-\bar{e}^{\T}\bar{e} \geq 0$, implying that $\bar{e} = 0$. In other words, if the input constraints are \emph{strictly feasible}, then \eqref{Eq:VI} is an exact error-zeroing design specification. A geometric interpretation of \eqref{Eq:VI} uses the \emph{normal cone} of the set $\Gamma$ at $\bar{\eta} \in \Gamma$, defined as
\[
\mathcal{N}^{P}_{\Gamma}(\bar{\eta}) \define \setdef{d \in \real^{p}}{\inner{d}{\eta-\bar{\eta}}{P} \leq 0\,\,\,\text{for all}\,\,\,\eta \in \Gamma}.
\]
With this, \eqref{Eq:VI} can be equivalently expressed as $-\bar{e} = -\pi(K\bar{\eta},w) \in \mathcal{N}^{P}_{\Gamma}(\bar{\eta})$, as illustrated in Figure \ref{Fig:Normal}. The interpretation of Figure \ref{Fig:Normal} is that from the point $\bar{\eta}$, any further attempt to adjust in the direction $-\bar{e} = -\pi(K\bar{\eta},w)$ {\tb would} result in constraint violation.

%

\begin{figure}[t!]
\centering
\includegraphics[width=0.65\columnwidth]{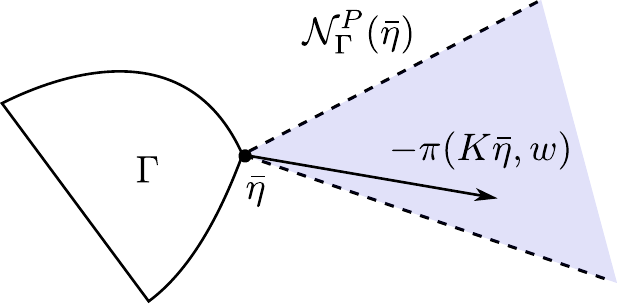}
\caption{Illustration of constrained error-zeroing specification.}
\label{Fig:Normal}
\end{figure}

\medskip

\begin{remark}[\bf Minimization Interpretation]\label{Rem:Min}
To see how else the design specification \eqref{Eq:VI} could arise, suppose that $\epsilon_k = h_{\epsilon}(x_k,u_k,w)$ is a measured tracking error of interest for the system \eqref{Eq:Plant}, with associated equilibrium mapping $\bar{\epsilon} = \pi_{\epsilon}(\bar{u},w)$ defined similar to \eqref{Eq:DefofPi}. Consider the steady-state minimization problem
\begin{equation}\label{Eq:Minimization}
\minimize_{\bar{u} \in \mathcal{C}}\,\, J(\bar{\epsilon}) \quad \Leftrightarrow \quad \minimize_{\bar{\eta} \in \Gamma} \,\,J(\pi_{\epsilon}(K\bar{\eta},w))
\end{equation}
where $\map{J}{\real^p}{\real}$ is a class $C^1$ convex and positive definite function. Critical points of this (generally, non-convex) problem are determined by the inclusion
\begin{equation}\label{Eq:CriticalPoint}
-K^{\sf T}\frac{\partial \pi_{\epsilon}}{\partial \bar{u}}(K\bar{\eta},w)^{\sf T}\nabla J(\pi_{\epsilon}(K\bar{\eta},w)) \in \mathcal{N}_{\Gamma}(\bar{\eta}).
\end{equation}
If we define the error signal $e_k \define K^{\sf T}\frac{\partial \pi_{\epsilon}}{\partial \bar{u}}(u_k,w)^{\T}\nabla J(\epsilon_k)$, then the inclusion \eqref{Eq:CriticalPoint} is precisely the variational inequality \eqref{Eq:VI}. {\tb Thus, the problem of steady-state minimization of a function of a tracking error can be interpreted as a special case of the specification \eqref{Eq:VI}. This perspective connects our approach directly with recent ideas in autonomous and feedback-based optimization; see \cite{LSPL-JWSP-EM:18l, MC-JWSP-AB:19c, AH-SB-GH-FD:21} for recent contributions. We note however that this interpretation via minimization is not necessary; \eqref{Eq:VI} can be directly interpreted as a generalization of a traditional perfect asymptotic tracking specification.}
\hfill \oprocend
\end{remark}
 
\subsection{The Projection (Forward-Backward) Algorithm}

The error-zeroing specification \eqref{Eq:VI} is equivalent to the so-called natural equation \cite[Chp. 1.5]{FF-JSP:03}
\begin{equation}\label{Eq:ProjectedEquation}
\bar{\eta} = \Proj_{\Gamma}^{P}(\bar{\eta} - \alpha \pi(K\bar{\eta},w))
\end{equation}
for any $\alpha > 0$, where $\map{\Proj_{\Gamma}^{P}}{\real^m}{\Gamma}$ defined as
\begin{equation}\label{Eq:Projection}
\Proj_{\Gamma}^{P}(\eta) = \argmin_{\nu \in \Gamma}\,\,\|\eta-\nu\|_{P}
\end{equation}
is the \emph{projection operator}, which yields the closest point to $\eta$ in $\Gamma$ measured in the norm $\|x\|_{P} = \sqrt{x^{\T}Px}$ induced by $P \succ 0$.
%
%
%
The equation \eqref{Eq:ProjectedEquation} leads immediately to the classic \emph{projection} or \emph{forward-backward splitting} algorithm \cite[Section 25.3]{HB-PC:11}.
\begin{equation}\label{Eq:FB}
\eta_{k+1} = (1-\lambda)\eta_k + \lambda\Proj_{\Gamma}^{P}(\eta_k-\alpha\pi(K\eta_k,w))
\end{equation}
for solving the variational inequality $\mathsf{VI}_{P}(\Gamma,\pi \circ K)$, where $\lambda \in (0,1)$ is a damping parameter. We summarize some well-known conditions which ensure exponential stability of the iteration \eqref{Eq:FB} to a unique equilibrium satisfying \eqref{Eq:ProjectedEquation}.

\smallskip

\begin{proposition}\label{Prop:FB}{(\bf Equilibrium and Contraction Properties of Forward-Backward Algorithm)}
Let $w \in \mathcal{W}$. If $\eta \mapsto \pi(K\eta,w)$ is $\mu$-strongly monotone on $\Gamma$ with respect to $\inner{\cdot}{\cdot}{P}$, i.e., if
\[
\inner{\pi(K\eta,w)-\pi(K\eta^{\prime},w)}{\eta-\eta^{\prime}}{P} \geq \mu \|\eta-\eta^{\prime}\|_{P}^2
\]
for some $\mu > 0$ and all $\eta,\eta^{\prime} \in \Gamma$, then \eqref{Eq:FB} possesses a unique equilibrium point $\bar{\eta} \in \Gamma$ satisfying \eqref{Eq:ProjectedEquation}. If $\eta \mapsto \pi(K\eta,w)$ is additionally $L$-Lipschitz continuous on $\Gamma$ with respect to the norm $\|\cdot\|_{P}$, and $\alpha$ is selected such that $\alpha \in (0,2\mu/L^2)$, then the following statements hold:
\begin{enumerate}
\item \label{Prop:FB-1} the foward-backward operator 
\begin{equation}\label{Eq:FBOperator}
\map{\Phi}{\Gamma}{\Gamma}, \qquad \Phi(\eta) = \Proj_{\Gamma}^{P}(\eta-\alpha\pi(K\eta,w))
\end{equation}
is a contraction mapping on $\Gamma$, satisfying
\[
\|\Phi(\eta)-\Phi(\eta^{\prime})\|_{P} \leq c_{\rm fb}\|\eta-\eta^{\prime}\|_{P},\quad \eta,\eta^{\prime} \in \Gamma,
\]
where $c_{\rm fb} = \sqrt{1-2\alpha \mu + \alpha^2 L^2} \in [0,1)$.
\item \label{Prop:FB-2} the damped forward-backward operator
\[
\map{\Phi_{\rm d}}{\Gamma}{\Gamma}, \qquad \Phi_{\rm d}(\eta) = (1-\lambda)\eta + \lambda \Phi(\eta)
\]
is a contraction mapping on $\Gamma$, satisfying
\[
\|\Phi_{\rm d}(\eta)-\Phi_{\rm d}(\eta^{\prime})\|_{P} \leq c_{\rm dfb}\|\eta-\eta^{\prime}\|_{P}.
\]
for all $\eta,\eta^{\prime} \in \Gamma$, where $c_{\rm dfb} = 1-\lambda(1-c_{\rm fb}) \in (0,1)$.
\end{enumerate}
\end{proposition}

\begin{proof}
{\tb The existence/uniqueness statement is \cite[Theorem 2.3.3]{FF-JSP:03}. The proof of (i) requires only minor modifications of the proof of \cite[Theorem 12.1.2]{FF-JSP:03b}, in which one uses the inner product $\inner{\cdot}{\cdot}{P}$ in place of $\inner{\cdot}{\cdot}{2}$ with step size $D = \tfrac{1}{\alpha}I$. The proof of (ii) then follows immediately from (i).}
\end{proof}

\section{Damped Projected Integral Control and Low-Gain Stability Result}
\label{Sec:Main}

\subsection{Damped Projected Integral Control}
\label{Sec:DPIC}

We propose adapting the forward-backward splitting algorithm \eqref{Eq:FB} as an integral feedback controller for enforcing the error-zeroing specification \eqref{Eq:VI}. Specifically, we propose the \emph{damped projected integral (DP-I) controller}
\begin{subequations}\label{Eq:DPIC}
\begin{align}
\eta_{k+1} &= (1-\lambda)\eta_{k} + \lambda \Proj^{P}_{\Gamma}(\eta_{k} - \tfrac{T_{\rm s}}{T_{\rm i}} e_k)\\
u_{k} &= K\eta_{k}
\end{align}
\end{subequations}
where $T_{\rm i} > 0$ is the integral time constant. We make several observations regarding \eqref{Eq:DPIC}:

\begin{enumerate}\itemsep=2pt
\item \emph{Constrained Error-Zeroing:} If \eqref{Eq:DPIC} is in equilibrium with the plant \eqref{Eq:Plant}, then it is immediate from \eqref{Eq:ProjectedEquation} that $\bar{\eta} \in \mathsf{VI}_{P}(\Gamma,\pi \circ K)$, {\tb which is precisely the constrained error-zeroing specification \eqref{Eq:VI}.}
\item \emph{Input Constraint Satisfaction \& Windup:} If $\eta_{k} \in \Gamma$, then $\eta_{k+1} \in \Gamma$, since by \eqref{Eq:DPIC} $\eta_{k+1}$ is a convex combination of two points in $\Gamma$. Therefore, $u_k = K\eta_k \in \mathcal{C}$ at all points in time. As a result, \eqref{Eq:DPIC} will never suffer from traditional integrator windup, {\tb as the controller output and plant output will always be in agreement. Note however that, rather than $\eta$ being conditionally frozen, the controller state $\eta_k$ may still change, moving along the boundary of the set $\Gamma$; see \cite{JT-JPH:09,AH-FD-AT:20,PS-SAB-CD-SDC:19} for related ideas in a continuous-time anti-windup design context.}

\item \emph{Reduction to Classical Integral Control:} If $\eta_k \in \Gamma$ and $\eta_{k} - \tfrac{T_{\rm s}}{T_{\rm i}} e_k \in \Gamma$, then the update \eqref{Eq:DPIC} reduces to 
\begin{equation}\label{Eq:LinearIntegralControl}
\eta_{k+1} = \eta_{k} - \tfrac{T_{\rm s}}{T_{\rm i}^{\prime}} e_{k}, \qquad u_k = K\eta_k
\end{equation}
where $T_{\rm i}^{\prime} = \lambda/T_{\rm i}$. Thus, when constraints are not encountered, \eqref{Eq:DPIC} reduces to the integral controller \eqref{Eq:LTIIntegral}.
\item \emph{Computation of Projection:} The projection in \eqref{Eq:DPIC} requires the solution of the convex optimization problem \eqref{Eq:Projection}, but need only be computed at step $k$ if $\eta_k - \frac{T_{\rm s}}{T_{\rm i}}e_k \notin \Gamma$. {\tb Projections onto many types of constraint sets are computable in closed-form  \cite[App. B]{AB:17}.}
\item \emph{Alternative Controller:} The controller \eqref{Eq:DPIC} is based on the natural equation associated with the inequality \eqref{Eq:VI}. If one instead uses a \emph{skewed} natural equation (see \cite[Chp. 1.5]{FF-JSP:03}), one can arrive at the alternative update law
\[
\eta_{k+1} = (1-\lambda)\eta_{k} + \lambda \Proj^{I_r}_{\Gamma}(\eta_{k} - \tfrac{T_{\rm s}}{T_{\rm i}} P^{-1} e_k),
\]
where the projection is now with respect to the standard Euclidean norm. In what follows though, we proceed with the formulation \eqref{Eq:DPIC}, mostly due to point (iii) above.
\end{enumerate}

%
%

\subsection{Low-Gain Stability with DP-I Control}

{\tb The closed-loop system is the interconnection of the plant \eqref{Eq:Plant} and the controller \eqref{Eq:DPIC}; we can now state our main result. }

\smallskip

\begin{theorem}[\bf Low-Gain Stability with DP-I Control]\label{Thm:DPIC}
Consider the plant \eqref{Eq:Plant} under Assumptions \ref{Ass:Plant-0}--\ref{Ass:Plant-3} with the DP-I controller \eqref{Eq:DPIC}. Suppose that there exists a matrix $P \succ 0$ and constants $\mu, L > 0$ such that $\eta \mapsto \pi(K\eta,w)$ is $\mu$-strongly monotone and $L$-Lipschitz continuous on $\Gamma$ with respect to $\inner{\cdot}{\cdot}{P}$, uniformly in $w \in \mathcal{W}$. Define $T_{\rm i}^{\star} \define T_{\rm s}L^2/2\mu$. Then for any $T_{\rm i} \in (T_{\rm i}^{\star},\infty)$, there exists $\lambda^{\star} \in (0,1)$ such that for any $\lambda \in (0,\lambda^{\star})$ and any $w \in \mathcal{W}$, the closed-loop system possesses an exponentially stable equilibrium point $(\bar{x},\bar{\eta}) \in \mathcal{X} \times \Gamma$ and the pair $(\bar{e},\bar{\eta}) = (\pi(K\bar{\eta},w),\bar{\eta})$ satisfies the error-zeroing specification \eqref{Eq:VI}.
\end{theorem}


\smallskip

To interpret the conditions in Theorem \ref{Thm:DPIC}, consider again the LTI case \eqref{Eq:LTIPi}. The condition for strong monotonicity requires that there exist $P \succ 0$ satisfying $(G(1)K)^{\T}P + PG(1)K \succ 0$, which is equivalent to the matrix $-G(1)K$ being Hurwitz stable; this is precisely Davison's classical condition, as described in Section \ref{Sec:Introduction}. The main condition required in Theorem \ref{Thm:DPIC} is that of strong monotonicity of the mapping $\eta \mapsto \pi(K\eta,w)$. {\tb As shown in \cite{JWSP:20a}, the same condition is sufficient for stability of low-gain integral control applied to continuous-time nonlinear systems in the absence of constraints; we further refer the reader to \cite[Sec. IV]{JWSP:20a} for a discussion of how this main condition can be checked computationally via semidefinite programming. }


\smallskip

\begin{pfof}{Theorem \ref{Thm:DPIC}}
The proof is based on a composite Lyapunov construction, and is divided into five steps.

\emph{Step \#1 \textemdash{} Equilibrium and Error Equations:} Let $w \in \mathcal{W}$ and set $\alpha \define T_{\rm s}/T_{\rm i}$. Equilibria $(\bar{x},\bar{\eta})$ are characterized by
\begin{equation}\label{Eq:ClosedLoopEq}
\begin{aligned}
\bar{x} &= f(\bar{x},\bar{u},w), \quad &\bar{\eta} &= \Proj_{\Gamma}^{P}(\bar{\eta}-\alpha \bar{e})\\
\bar{e} &= h(\bar{x},\bar{\eta},w), \quad &\bar{u} &= K\bar{\eta}.
\end{aligned}
\end{equation}
If such an equilibrium exists, then necessarily $\bar{\eta} \in \Gamma$, and hence $\bar{u} = K\bar{\eta} \in \mathcal{C}$. Given any such $\bar{u}$, it follows from \ref{Ass:Plant-1} that the first equation in \eqref{Eq:ClosedLoopEq} can be solved for $\bar{x} = \pi_{x}(\bar{u},w)$; together, \ref{Ass:Plant-1}/\ref{Ass:Plant-3} imply that $\bar{x}$ is isolated. Eliminating $\bar{x}$ and $\bar{e}$, we obtain the reduced equilibrium equation
\begin{equation}\label{Eq:EquilibriumProj}
\bar{\eta} = \Proj_{\Gamma}^{P}(\bar{\eta}-\alpha \pi(K\bar{\eta},w)) = \Phi(\bar{\eta}) = \Phi_{\rm d}(\bar{\eta})
\end{equation}
which is equivalent to the error-zeroing specification \eqref{Eq:VI}. Since $\eta \mapsto \pi(K\eta,w)$ is $\mu$-strongly monotone on $\Gamma$ uniformly in $w$, and $\Gamma$ is closed, convex, and non-empty, $\mathsf{VI}_{P}(\Gamma,\pi \circ K)$ admits a unique solution \cite[Theorem 2.3.3]{FF-JSP:03}. We conclude that the closed-loop system possess a unique equilibrium point $(\bar{x},\bar{\eta}) \in \mathcal{X} \times \Gamma$ with $\bar{e} = h(\bar{x},K\bar{\eta},w) = \pi(K\bar{\eta},w)$ and control $\bar{u} = K\bar{\eta} \in \mathcal{C}$. Consider the change of state variable
\[
\begin{aligned}
\xi_k &\define x_k - \pi_x(K\eta_k,w).
\end{aligned}
\]
With this, the dynamics \eqref{Eq:Plant},\eqref{Eq:DPIC} become
\begin{equation}\label{Eq:NewCoordDynamics}
\begin{aligned}
\xi_{k+1} &= f(\xi_{k} + \pi_{x}(K\eta_k,w), K\eta_k,w) - \pi_x(K\eta_{k+1},w)\\
e_k &= h(\xi_k+\pi_x(K\eta_k,w),K\eta_k,w)\\
\eta_{k+1} &= (1-\lambda)\eta_k + \lambda \Proj_{\Gamma}^{P}(\eta_k - \alpha e_k),
\end{aligned}
\end{equation}
and the equilibrium point of interest is $(\xi,\eta) = (0,\bar{\eta})$. 

\medskip

\emph{Step \#2 \textemdash{} Analyzing the Slow Dynamics:} Let $V_{\rm s}(\eta) = \|\eta-\bar{\eta}\|_{P}^2$. Using $\Phi$ and $\Phi_{\rm d}$ from Proposition \ref{Prop:FB}, we compute that
\[
\begin{aligned}
V_{\rm s}(\eta_{k+1})^{\tfrac{1}{2}} &= \|(1-\lambda)\eta_k + \lambda \Proj_{\Gamma}^{P}(\eta_k - \alpha e_k) - \bar{\eta}\|_{P}\\
&= \|(1-\lambda)\eta_{k} + \lambda \Phi(\eta_k) - \bar{\eta} \\
&\quad \qquad + \lambda (\Proj_{\Gamma}^{P}(\eta_k-\alpha e_k)-\Phi(\eta_k))\|_{P}\\
&= \|\Phi_{\rm d}(\eta_k) - \Phi_{\rm d}(\bar{\eta})\\
&\qquad + \lambda (\Proj_{\Gamma}^{P}(\eta_k-\alpha e_k)-\Phi(\eta_k))\|_{P}\\
&\leq c_{\rm dfb}\|\eta_{k}-\bar{\eta}\|_{P} + \lambda \|\delta\|_{P}
\end{aligned}
\]
where $\delta = \Proj_{\Gamma}^{P}(\eta_k-\alpha e_k)-\Phi(\eta_k)$ {\tb and $c_{\rm dfb}$ is defined in Proposition \ref{Prop:FB}}. To bound $\|\delta\|_{P}$ we compute that
\[
\begin{aligned}
\|\delta\|_{P}^2 &= \|\Proj_{\Gamma}^{P}(\eta_k-\alpha e_k)-\Proj_{\Gamma}^{P}(\eta_k-\alpha\pi(K\eta_k,w))\|_{P}^2\\
&\leq \alpha^2\|e_k-\pi(K\eta_k,w)\|_{P}^2\\
&= \alpha^2 \|h(\xi_k+\pi_x(K\eta_k,w),K\eta_k,w)\\
&\qquad \qquad \qquad \qquad - h(\pi_x(K\eta_k,w),K\eta_k,w)\|_P^2\\
&\leq \alpha^2\lambda_{\rm max}(P) L_h^2 \|\xi_k\|_2^2
\end{aligned}
\]
where $L_h$ is the Lipschitz constant of $h$. Combining the above, with $\Delta V_{\rm s} = V_{\rm s}(\eta_{k+1}) - V_{\rm s}(\eta_k)$, one finds that along trajectories of \eqref{Eq:NewCoordDynamics} it holds that
\[
\begin{aligned}
\Delta V_{\rm s} &\leq (c_{\rm dfb}^2-1) \|\eta_k-\bar{\eta}\|_{P}^2 + \lambda_{\rm max}(P)L_h^2\alpha^2\lambda^2\|\xi_k\|_2^2\\
&\qquad + 2\lambda_{\rm max}(P)^{\tfrac{1}{2}}\alpha L_h\lambda c_{\rm dfb}\|\eta_k-\bar{\eta}\|_{P}\|\xi_k\|_2\\
&= \zeta_k^{\T}Q_{\rm s}\zeta_k
\end{aligned}
\]
where $\zeta_k = \mathrm{col}(\|\xi_k\|_2,\|\eta_k-\bar{\eta}\|_{P})$ and
\[
Q_{\rm s} = \begin{bmatrix}
q_1\lambda^2 & q_2\lambda \\
q_2\lambda  & c_{\rm dfb}^2-1
\end{bmatrix}, \quad \begin{aligned}
q_1 = \lambda_{\rm max}(P)\alpha^2 L_h^2\\
q_2 = \alpha \lambda_{\rm max}(P)^{1/2}L_h c_{\rm dfb}.
\end{aligned}
\]

\medskip

\emph{Step \#3 \textemdash{} Bounding $\|\eta_{k+1}-\eta_{k}\|_{P}$:} We compute using the triangle inequality that
\begin{equation}\label{Eq:etaketak}
\begin{aligned}
\|\eta_{k+1}-\eta_{k}\|_{P} &\leq \|\eta_{k+1} - \Phi_{\rm d}(\eta_k)\|_{P}\\
&\qquad  + \|\Phi_{\rm d}(\eta_k) - \eta_k\|_{P}.
\end{aligned}
\end{equation}
Using our previous calculations, the first term in \eqref{Eq:etaketak} can be bounded as
\[
\begin{aligned}
\|\eta_{k+1} - \Phi_{\rm d}(\eta_k)\|_{P} &= \lambda\|\Proj_{\Gamma}^{P}(\eta_k - \alpha e_k) - \Phi(\eta_k)\|_{P}\\
&= \lambda \|\delta\|_{P}\\
&\leq \lambda \alpha L_h \lambda_{\rm max}(P)^{1/2} \|\xi_k\|_2.
\end{aligned}
\]
To bound the second term in \eqref{Eq:etaketak}, it follows from Proposition \ref{Prop:FB} and the triangle inequality that
\begin{equation}\label{Eq:PhiDetaketak}
\begin{aligned}
\|\Phi_{\rm d}(\eta_k)-\eta_k\|_{P} &= \lambda\|\Phi(\eta_k)-\eta_k\|_{P}\\
&= \lambda\|(\eta_k-\bar{\eta}) - (\Phi(\eta_k) - \bar{\eta})\|_{P}\\
&\leq \lambda(1+c_{\rm fb})\|\eta_k-\bar{\eta}\|_{P}.
\end{aligned}
\end{equation}
Putting things together we obtain
\begin{equation}\label{Eq:EtaEtaBound}
\begin{aligned}
\|\eta_{k+1}-\eta_{k}\|_{P} &\leq \lambda L_h \alpha \lambda_{\rm max}(P)^{1/2} \|\xi_k\|_2\\
&\qquad  + \lambda(1+c_{\rm fb})\|\eta_k-\bar{\eta}\|_{P}.
\end{aligned}
\end{equation}

\medskip

\emph{Step \#4 \textemdash{} Analyzing the Fast Dynamics:} Define the deviation vector field $\map{g}{\real^n \times \mathcal{U} \times \mathcal{W}}{\real^n}$ by
\[
\begin{aligned}
g(\xi,u,w) &= f(\xi+\pi_x(u,w),u,w) - f(\pi_x(u,w),u,w)\\
&= f(\xi+\pi_x(u,w),u,w) - \pi_x(u,w).
\end{aligned}
\]
Under Assumptions \ref{Ass:Plant-0}--\ref{Ass:Plant-3}, the conditions of a converse Lyapunov theorem for exponential stability are satisfied (see \cite[Thm A1]{JWSP:21a-extended}): there exists a set $\mathcal{Z}$ containing the origin in its interior, positive constants $c_1, c_2, c_3, c_4 > 0$, $\rho_{\rm f} \in [0,1)$, and a continuous function
\[
\map{V_{\rm f}}{\mathcal{Z} \times \mathcal{U} \times \mathcal{W}}{\real_{\geq 0}}, \quad (\xi,u,w) \mapsto V_{\rm f}(\xi,u,w)
\]
satisfying the Lyapunov conditions
\begin{subequations}\label{Eq:LyapFast}
\begin{align*}
&c_1 \|\xi\|_2^2 \leq V_{\rm f}(\xi,u,w) \leq c_2 \|\xi\|_2^2\\
%
%
&V_{\rm f}(g(\xi,u,w),u,w) - V_{\rm f}(\xi,u,w) \leq -\rho_{\rm f} \|\xi\|_2^2\\
%
%
&|V_{\rm f}(\xi,u,w) - V_{\rm f}(\xi^{\prime},u,w)| \leq c_3 (\|\xi\|_2 + \|\xi^{\prime}\|_2)\|\xi-\xi^{\prime}\|_2\\
%
%
&|V_{\rm f}(\xi,u,w) - V_{\rm f}(\xi,u^{\prime},w)| \leq c_4\|\xi\|_2^2\|u-u^{\prime}\|_2
\end{align*}
\end{subequations}
for all $\xi,\xi^{\prime} \in \mathcal{Z}$, all $u,u^{\prime} \in \mathcal{U}$, and all $w \in \mathcal{W}$. Let 
\[
\begin{aligned}
\Delta V_{\rm f} &= V_{\rm f}(\xi_{k+1},\eta_{k+1},w) - V_{\rm f}(\xi_k,\eta_k,w)\\
\end{aligned}
\]
denote the increment of $V_{\rm f}$ along trajectories of \eqref{Eq:NewCoordDynamics}. 
While we suppress the details due to space limitations\footnote{Available in extended version of this paper  \cite{JWSP:21a-extended}.}, one may use the Lyapunov properties to conclude that there exists $r > 0$ such that $\Delta V_{\rm f} \leq \zeta_k Q_{\rm f}\zeta_k$ holds for all $(\xi_k,\eta_k,w) \in \mathcal{B}_{r}(0) \times \Gamma \times \mathcal{W}$, where 
\[
\begin{aligned}
Q_{\rm f} &= \begin{bmatrix}-\rho_{\rm f} + k_1\lambda^2 + (k_2+k_6)\lambda & k_3\lambda^2 + (k_4+k_7)\lambda\\
k_3\lambda^2 + (k_4+k_7)\lambda & k_5\lambda^2
\end{bmatrix}
\end{aligned}
\]
and where the constants $k_1$ through $k_7$ are positive and independent of $\lambda$. 

\emph{Step \#5 -- Putting the Pieces Together:} Define the composite Lyapunov candidate $V(\xi,\eta,w) = V_{\rm s}(\eta) + V_{\rm f}(\xi,\eta,w)$. Along trajectories of \eqref{Eq:NewCoordDynamics}, we combine the previous inequalities to compute that
\[
\Delta V = V(\xi_{k+1},\eta_{k+1},w) - V(\xi_{k},\eta_{k},w) \leq \zeta_k^{\T}Q\zeta_k
\]
holds for all $(\xi_k,\eta_k,w) \in \mathcal{B}_{r}(0) \times \Gamma \times \mathcal{W}$, where
\[
Q = \begin{bmatrix}-\rho_{\rm f} + (k_1+q_1)\lambda^2 + \tilde{k}_2\lambda & k_3\lambda^2 + \tilde{k}_4\lambda\\
k_3\lambda^2 + \tilde{k}_4\lambda & -(1-c_{\rm dfb}^2) + k_5\lambda^2
\end{bmatrix}
\]
and where for compactness we set $\tilde{k}_2 = k_2 + k_6$ and $\tilde{k}_4 = k_4 + k_7 + q_2$. Note that the $(1,1)$ element of $Q$ is negative and $\mathcal{O}(1)$ as $\lambda \to 0$\footnote{For a function $\map{g}{\real}{\real}$ which is positive definite with respect to $0$, a function $\map{f}{\real}{\real}$ is $\mathcal{O}(g(\lambda))$ as $\lambda \to 0$ if $\lim_{\lambda \to 0}|f(\lambda)|/g(\lambda) < \infty$.}. From Proposition \ref{Prop:FB}
\[
1-c_{\rm dfb}^2 = 2\lambda(1-c_{\rm fb}) - \lambda^2(1-c_{\rm fb})^2,
\]
with $c_{\rm fb} \in (0,1)$, and therefore the $(2,2)$ element of $Q$ is negative and $\mathcal{O}(\lambda)$ as $\lambda \to 0$. Since the off diagonal elements are $\mathcal{O}(\lambda)$ as $\lambda \to 0$, it is straightforward to argue that there exists some $\lambda^{\star} > 0$ such that $Q \prec 0$ for all $\lambda \in (0,\lambda^{\star})$. It follows that there exists $\varepsilon > 0$ such that $\Delta V(\xi_k,\eta_k,w) \leq -\varepsilon V(\xi_k,\eta_k,w)$ for all $(\xi_k,\eta_k,w) \in \mathcal{B}_{r}(0) \times \Gamma \times \mathcal{W}$. Standard arguments (e.g., \cite[Thm. 13.2]{WMH-VC:19} now complete the proof.
\end{pfof}


{\tb

\section{Example: Four-Tank Process}
\label{Sec:Simulation}

We illustrate our approach with an application to sampled-data control of nonlinear process describing water flow between four interconnected tanks; see \cite{KHJ:00} for a schematic. With state $h \in \real_{>0}^4$ describing the water levels in the four tanks, and inputs $u \in \real_{\geq 0}^2$ being flow rates for the two pumps, the continuous-time system dynamics can be expressed as
\begin{equation}\label{Eq:Tank}
\dot{h} = \mathcal{A}\phi(h) + \mathcal{B}u,\, \quad y = \mathrm{col}(h_1,h_2)
\end{equation}
where $\phi(h) = \mathrm{col}(\sqrt{2gh_1},\sqrt{2gh_2},\sqrt{2gh_3},\sqrt{2gh_4})$ and
\[
\mathcal{A} = \begin{bmatrix}
-\tfrac{a_1}{A_1} & 0 & \tfrac{a_3}{A_1} & 0\\
0 & -\tfrac{a_2}{A_2} & 0 & \tfrac{a_4}{A_2}\\
0 & 0 & -\tfrac{a_3}{A_3} & 0\\
0 & 0 & 0 & -\tfrac{a_4}{A_4}
\end{bmatrix},\,\, \mathcal{B} = \begin{bmatrix}
\tfrac{\gamma_1}{A_1} & 0\\
0 & \tfrac{\gamma_2}{A_2}\\
0 & \tfrac{1-\gamma_2}{A_3}\\
\tfrac{1-\gamma_1}{A_4} & 0
\end{bmatrix},
\]
with parameters as given in \cite{KHJ:00}. With $u^{\star} = (32.64,32.64)$, the point $h^{\star} = (10,10,5.38,5.38)$ is an exponentially stable equilibrium, which we consider as the nominal operating point for the system. While we omit the details due to space limitations, one can find appropriate sets $\mathcal{X}$ and $\mathcal{U}$ around this operating point such that \ref{Ass:Plant-0}--\ref{Ass:Plant-3} hold for the ideally discretized model associated with \eqref{Eq:Tank}. The control objective is to regulate the water levels $h_1, h_2$ in the two lower tanks to specified set-points $r_1,r_2$, and we therefore take $e = (h_1-r_1,h_2-r_2)$ as the error signal of interest. The control inputs $u_1,u_2$ are constrained to lie in set
\[
\mathcal{C} = \setdef{(u_1,u_2)}{u_1 \in [0,45],\,\,u_2 \in [0,45],\,\, u_1 + u_2 \leq 85}.
\]
modelling individual and total flow rate constraints for the two pumps. Straightforward computations show that \eqref{Eq:DefofPi} is
\[
\pi(\bar{u}) = \frac{1}{2g}\mathrm{diag}(\Pi \bar{u})\Pi \bar{u}, \quad \Pi \define \begin{bmatrix}
\tfrac{\gamma_1}{a_1} & \tfrac{1-\gamma_2}{a_1}\\
\tfrac{1-\gamma_1}{a_2} & \tfrac{\gamma_2}{a_2}
\end{bmatrix}.
\]
Selecting $K = \Pi^{-1}$ ensures that the monotonicity and Lipschitz conditions in Theorem \ref{Thm:DPIC} are satisfied with $P = I_2$ on a large domain containing the nominal operating point. The remaining parameters for the DP-I controller \eqref{Eq:DPIC} are selected as $T_{\rm s} = 10$s, $T_{\rm i} = 15$s, and $\lambda = 0.95$. Figure \ref{Fig:Tank} shows the closed-loop response to sequential set-point changes for the water levels in the first two tanks. During the first two reference changes, the constraints $u \in \mathcal{C}$ are strictly feasible, and hence exact reference tracking is observed in Figure \ref{Fig:Tank1}. After the next two reference changes, the constraints $u_1 \in [0,45]$ and $u_1 + u_2 \leq 85$ are both encountered. Exact tracking is no longer possible, but input constraints are maintained and the closed-loop system remains stable and reaches a steady-state satisfying the variational inequality \eqref{Eq:VI}.

\begin{figure}[ht!]
\centering
\begin{subfigure}{0.99\linewidth}
\includegraphics[width=\linewidth]{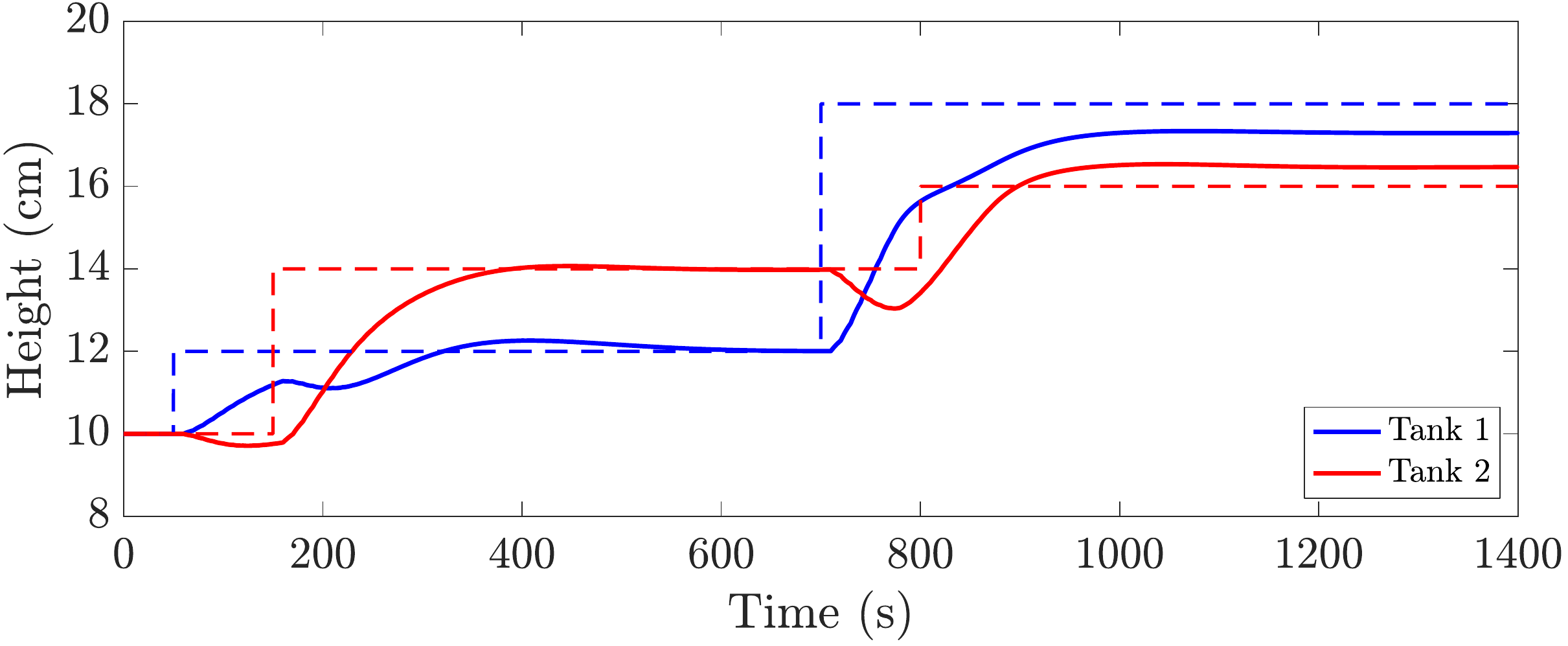}
\caption{Tank water levels; dashed lines denote reference set-points.}
\label{Fig:Tank1}
\end{subfigure}\\
%
%
\smallskip
\begin{subfigure}{0.99\linewidth}
\includegraphics[width=\linewidth]{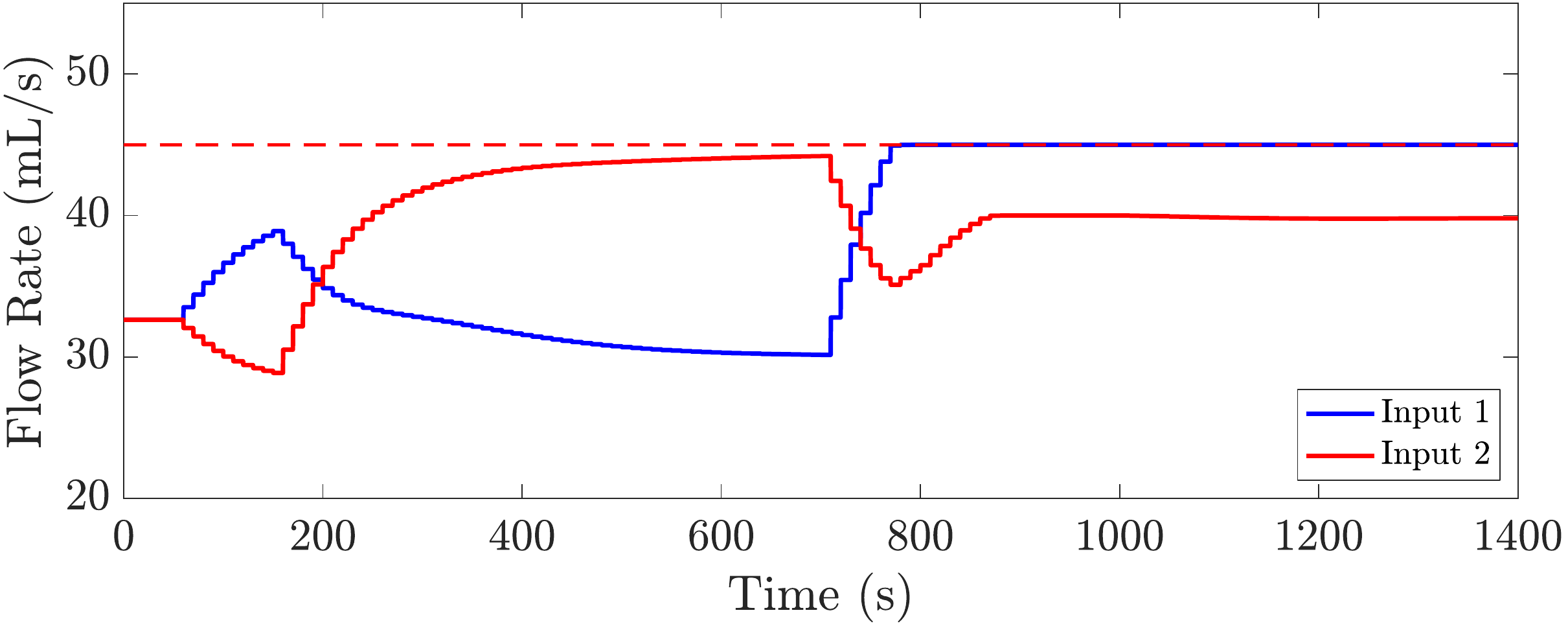}
\caption{Control signals; dashed line denotes individual pump limit.}
\label{Fig:Tank2}
\end{subfigure}
\caption{{\tb Closed-loop response of 4-tank system.}}
\label{Fig:Tank}
\end{figure}

}

\section{Conclusions}
\label{Sec:Conclusion}

We have formulated an approximate tracking specification as a variational inequality, and designed projected integral controller to meet this specification while maintaining arbitrary convex constraints on the input signal at all times. In the absence of constraints, the approximate tracking specification reduces to an exact tracking specification, and the projected integral controller reduces to a classical integral controller. {\tb The controller inherits what is perhaps the most important stability property of traditional integral control, namely that under a monotonicity condition on the plant equilibrium mapping, closed-loop stability can be guaranteed when the plant is exponentially stable and the integral gain is sufficiently low.} Future work will consider the extension of this scheme to a projected PID controller, and the extension to more general discrete-time output-regulating controllers which admit a representation in so-called incremental form.

\renewcommand{\baselinestretch}{1}
\bibliographystyle{IEEEtran}

\bibliography{/Users/jwsimpso/GoogleDrive/JohnSVN/bib/brevalias,%
/Users/jwsimpso/GoogleDrive/JohnSVN/bib/Main,%
/Users/jwsimpso/GoogleDrive/JohnSVN/bib/JWSP,%
/Users/jwsimpso/GoogleDrive/JohnSVN/bib/New%
}


\end{document}